\documentclass[11pt,a4paper]{amsart}
\usepackage{subfiles}
\usepackage{hyperref}
\usepackage{mathrsfs}

\newenvironment{changemargin}[2]{%
\begin{list}{}{%
\setlength{\topsep}{0pt}%
\setlength{\leftmargin}{#1}%
\setlength{\rightmargin}{#2}%
\setlength{\listparindent}{\parindent}%
\setlength{\itemindent}{\parindent}%
\setlength{\parsep}{\parskip}%
}%
\item[]}
{\end{list}}

\hypersetup{urlcolor=blue, citecolor=blue, linkcolor=blue, colorlinks=true}

\usepackage{enumitem}
\usepackage{amsmath}
\usepackage{array}
\usepackage{amssymb}
\usepackage{mathtools}
\usepackage{faktor}
\usepackage{amsfonts}
\usepackage{microtype}
\usepackage{multicol}
\usepackage{subcaption}
\usepackage{xcolor}
\usepackage{stackengine,scalerel}
\usepackage{tikz}
\usetikzlibrary{arrows}
\usepackage{tikz-cd}
\usepackage{amsthm}
\usepackage{tcolorbox}
\usepackage[capitalise]{cleveref}
\usepackage[utf8]{inputenc}
\usepackage{csquotes}
\usepackage[english]{babel}

\newcolumntype{L}{>{\scriptstyle}l}
\newcolumntype{C}{>{\scriptstyle}c}
\newcolumntype{R}{>{\scriptstyle}r}

\DeclareMathAlphabet{\mathpzc}{OT1}{pzc}{m}{it}

\usepackage{newunicodechar}
\usepackage[style=alphabetic,backend=biber,maxnames=5,maxalphanames=5,doi=true,isbn=false,url=false]{biblatex}
\bibliography{literature.bib}
\renewbibmacro{in:}{}
\renewbibmacro*{issue+date}{%
  \ifboolexpr{not test {\iffieldundef{year}} or not test {\iffieldundef{issue}}}
    {\printtext[parens]{%
       \iffieldundef{issue}
         {\usebibmacro{date}}
         {\printfield{issue}%
          \setunit*{\addspace}%
          \usebibmacro{date}}}}
    {}%
  \newunit}
\let\oldtocsection=\tocsection
\let\oldtocsubsection=\tocsubsection
\let\oldtocsubsubsection=\tocsubsubsection
\renewcommand{\tocsection}[2]{\hspace{0em}\oldtocsection{#1}{#2}}
\renewcommand{\tocsubsection}[2]{\hspace{1em}\oldtocsubsection{#1}{#2}}
\renewcommand{\tocsubsubsection}[2]{\hspace{2em}\oldtocsubsubsection{#1}{#2}}
\newtheorem{bigthm}{Theorem}

\newtheorem{thm}{Theorem}[section]
\newtheorem*{thm*}{Theorem A (ii)'}
\newtheorem{lem}[thm]{Lemma}
\newtheorem{prop}[thm]{Proposition}
\newtheorem{cor}[thm]{Corollary}
\newtheorem{question}[thm]{Question}

\theoremstyle{definition}
\newtheorem{dfn}[thm]{Definition}

\theoremstyle{remark}
\newtheorem{example}[thm]{Example}

\newtheorem{rem}[thm]{Remark}

\setenumerate[1]{leftmargin=*,labelindent=0pt,label=(\roman*),}
\newcommand{\pref}[2]{\hyperref[#2]{#1 \ref*{#2}}}
\usepackage{todonotes}

\newcommand{\id}{\ensuremath{\operatorname{id}}}
\newcommand{\im}{\ensuremath{\operatorname{im}}}
\newcommand{\Spin}{{\ensuremath{\operatorname{Spin}}}}
\newcommand{\so}{{\ensuremath{\operatorname{SO}}}}
\newcommand{\ort}{{\ensuremath{\operatorname{O}}}}

\newcommand{\bo}{B\ort}
\newcommand{\bso}{B\so}
\newcommand{\bspin}{B\Spin}

\newcommand{\diff}{\ensuremath{\operatorname{Diff}}}

\newcommand{\too}{\longrightarrow}
\newcommand{\rk}{\mathrm{rank}\,}
\newcommand{\res}{\mathrm{res}}
\newcommand{\inn}{\mathrm{int}}

\newcommand{\rst}{\mathrm{rst}}

\newcommand{\orient}{\mathrm{or}}
\newcommand{\congarrow}{\overset{\cong}\longrightarrow}
\newcommand{\embeds}{\hookrightarrow}
\newcommand{\actson}{\curvearrowright}
\DeclareMathOperator{\inddiff}{\mathrm{inddiff}}

\newcommand{\calR}{\mathcal{R}}

\newcommand{\bbN}{\mathbb{N}}

\newcommand{\bbZ}{\mathbb{Z}}

\newcommand{\bbR}{\mathbb{R}}

\newcommand{\ahat}{\hat{\mathcal{A}}}

\newcommand{\cp}[1]{\mathbb{CP}^{#1}}

\newcommand{\op}{{\mathrm{op}}}

\newcommand{\ko}{{\mathrm{KO}}}
\newcommand{\pt}{{\mathrm{pt}}}

\newcommand{\dt}{\mathrm{dt}}
\newcommand{\tor}{\mathrm{tor}}

\begin{document}
\author[Georg Frenck]{Georg Frenck}
\email{\href{mailto:georg.frenck@math.uni-augsburg.de}{georg.frenck@math.uni-augsburg.de}}
\email{\href{mailto:math@frenck.net}{math@frenck.net}}
\urladdr{\href{http://frenck.net/Math}{Frenck.net/Math}}
\address{Universität Augsburg, Universitätsstr.~14, 86159 Augsburg, Germany}

\subjclass[2010]{53C21, 55R40, 57R75, 57R90, 58D17, 58D05.}
\keywords{Positive scalar curvature, spaces of metrics, manifolds, spin, totally nonspin, diffeomorphisms}

\thanks{G.~F.~is supported by the DFG (German Research Foundation) through RTG 2229 (281869850) and the Special Priority Programme SPP 2026. This version of the article has been accepted for publication, after peer review (when applicable) but is not the Version of Record and does not reflect post-acceptance improvements, or any corrections. The Version of Record is available online at: http://dx.doi.org/10.1007/s00209-023-03270-1}

\title[PSC-Metrics on totally nonspin Manifolds]{Spaces of Positive Scalar Curvature metrics on totally nonspin Manifolds with spin boundary}

\begin{abstract} 
In this article we study the space of positive scalar curvature metrics on totally nonspin manifolds with spin boundary. We prove that for such manifolds of certain dimensions, those spaces are not connected and have nontrivial fundamental group. Furthermore we show that a well-known propagation technique for detection results on spaces of positive scalar curvature metrics on spin manifolds ceases to work in the totally nonspin case.
\end{abstract}

\maketitle

\section{Introduction}
\noindent For a closed smooth manifold $M$ we denote by $\calR^+(M)$ the space of all Riemannian metrics of positive scalar curvature (psc-metrics) equipped with the $C^\infty$-topology. In recent years, there has been a lot of research activity around such spaces of metrics with lower curvature bounds, see e.g. \cite{HSS, berw, CSS, erw_psc2, a-hat-bundles, ahatblock}. It has been shown that in the presence of a spin structure, $\calR^+(M)$ has a lot of interesting topological features like nontrivial homotopy or homology groups. 

In this article, we deviate from this setup in two ways: First, we consider compact manifolds $M$ with non-empty  boundary $\partial M$ and we require the metrics in $\calR^+(M)$ to be cylindrical in a neighbourhood of the boundary (see \pref{Section}{sec:prelims-psc} for a precise definition). Second, we want $\partial M$ to be $\Spin$, whereas $M$ shall be \emph{totally nonspin}, i.e. its universal cover $\widetilde{M}$ does not admit a spin structure. Employing this discrepancy, we show that index-theoretic secondary obstructions on the boundary can be used to distinguish components and elements in the fundamental group of $\calR^+(M)$.

We also observe that a well-known propagation technique for positive scalar curvature metrics on spin manifolds ceases to work in the totally nonspin case. More precisely we show that there exist non-isotopic psc-metrics on the disc which become isotopic after extending them arbitrarily to a totally nonspin manifold.

\begin{center}\emph{Throughout this article we will assume that $M$ and $\partial M$ admit psc-metrics.}\end{center}

\subsection{Spaces of metrics on manifolds with boundary} Consider the map
\[\res\colon\calR^+(M)\to\calR^+(\partial M)\]
given by restricting a psc-metric to the boundary which is a Serre-fibration by \cite[Theorem 1.1]{ebertfrenck}. For $h\in\calR^+(\partial M)$ we define $\calR^+(M)_h\coloneqq \res^{-1}(h)$. One powerful tool for studying the space of psc-metrics on a spin manifold $N$ is the index-difference first introduced by Hitchin in \cite{hitchin_spinors}. After fixing a base-point $g\in\calR^+(N)_h$, it induces a map 
\[\inddiff_g\colon \pi_k(\calR^+(N)_h)\to \ko^{-\dim(N)-k-1}(\pt). \]
This map is often non-trivial (see for example \cite[Theorem A]{berw}) and our first main result is the following.
\begin{bigthm}\label{main:boundary}
	Let $M$ be a compact, oriented, totally nonspin manifold of dimension $\dim M=d\ge8$ with non-empty spin boundary $\partial M$.
	\begin{enumerate}
		\item If $d\equiv0,1(8)$ and $\partial M$ is $\Spin\times B\pi_1(\partial M)$-nullbordant\footnote{See \pref{Section}{sec:psc-recollection} for a definition.}, there exists an $h\in\calR^+(\partial M)$ such that $\calR^+(M)_h$ is non-empty and the following composition is surjective 
			\[\pi_1(\calR^+(M))\overset{\res}\too\pi_1(\calR^+(\partial M))\overset{\inddiff_h}{\too}\ko^{-d-1}(\pt)\cong\bbZ/2.\]
		\item For any $h\in\calR^+(\partial M)$ such that $\calR^+(M)_h$ is non-empty, the following composition is surjective:
			\[\pi_0(\calR^+(M))\overset{\res}\too\pi_0(\calR^+(\partial M))\overset{\inddiff_h}{\too}\ko^{-d}(\pt).\]
	\end{enumerate}
\end{bigthm}

\begin{rem}
	\begin{enumerate}
		\item The most prominent examples of totally nonspin manifolds are even-dimensional complex projective spaces. For $D^{4k}\subset \cp{2k}$ an embedded open disk, \pref{Theorem}{main:boundary} implies
		\begin{align*}
			&\pi_1(\calR^+(\cp{4n}\setminus D^{8n}))\not=1\quad\ \ \, \text{ for $n\ge1$}\\
			|&\pi_0(\calR^+(\cp{2n}\setminus D^{4n}))| = \infty \quad \text{ for $n\ge2$.}
		\end{align*}
		\item \pref{Theorem}{main:boundary} remains true, if we consider spaces of metrics with for example mean convex or minimal boundary instead of requiring product form near the boundary. By \cite[Corollary 40]{BH20} the space of psc-metrics on $M$ which satisfy one of these boundary conditions and whose restriction to the boundary is again of positive scalar curvature, is homotopy equivalent to $\calR^+(M)$. 
	\end{enumerate}
\end{rem}

\begin{rem}[State of the art]
To the best of the author's knowledge, \pref{Theorem}{main:boundary} is the first detection result concerning higher homotopy groups of the space of psc-metrics on a nonspin manifold. Furthermore, it is also the first such result about $\calR^+(M)$ for manifolds with boundary, where the manifold and its boundary are allowed to have isomorphic fundamental groups and can even be simply connected.
				
		 In the case where $M$ is nonspin, but its universal cover is $\Spin$, there are cases where $\calR^+(M)$ is not connected or has even infinitely many path components, see \cite{BG95, reiser, DGA, wermelinger, goodman, dessai, goodmanwermelinger}. For totally nonspin manifolds, Kastenholz--Reinhold give an example of a closed, totally nonspin manifold of dimension $6$ whose space of psc-metrics has infinitely many components in \cite{kastenholzreinhold}. 
		 
		 In \cite[Theorem 44]{BH20}, B\"ar--Hanke give examples of spin manifolds $M$ with boundary where $\calR^+(M)$ has a non-vanishing homotopy group. In their examples however, $M$ has infinite fundamental group and the induced map $\pi_1(\partial M)\to \pi_1(M)$ is not an isomorphism.
\end{rem}

\subsection{Gluing in psc-metrics on disks}
Let $N$ be a closed $d$-dimensional spin manifold. One nice feature of the index-difference map is the following \enquote{propagation principle}: let $D\subset N$ be an embedded disk and let $h$ be a psc-metric on $S^{d-1}$ such that $\calR^+(N\setminus \inn(D))_{h}\not=\emptyset\not=\calR^+(D^d)_{h}$. The additivity theorem for the index-difference \cite[Proposition 3.18]{berw} implies that
\[\inddiff_{g_0}([g_1]) = \inddiff_{g_0\cup g}([g_1\cup g]),\] 
for any $g\in\calR^+(N\setminus \inn(D))_{h}$ and any $g_0,g_1\in\calR^+(D^d)_{h}$. Therefore, non-isotopic psc-metrics on the disc which are detected by the index-difference remain non-isotopic after being extended arbitrarily to another spin manifold, implying one can propagate detection results. The following result states that in the totally nonspin case, this is no longer possible.

\begin{bigthm}\label{main:hitchin}
	Let $d\equiv0,1(8)$, $d\ge8$ and let $M$ be a closed, oriented $d$-dimensional, totally nonspin manifold with $D\subset M$ an embedded disc. Let $h\in\calR^+(S^{d-1})$ such that $\calR^+(M\setminus \inn(D))_{h}\not=\emptyset\not=\calR^+(D^d)_{h}$ and let $g\in\calR^+(M\setminus \inn(D))_{h}$ be arbitrary.	
	Then there exist metrics $g_0, g_1\in\calR^+(D^d)_{h}$ with $\inddiff_{g_0}(g_1)\not=0\in\ko^{-d-1}(\pt)\cong \bbZ/2$, such that $g_0\cup g$ and $g_1\cup g$ are isotopic. 
\end{bigthm}

\begin{rem}
\begin{enumerate}
	\item As mentioned above, \pref{Theorem}{main:hitchin} does not hold in the spin case. Even further, we have the following result (cf. \cite[Theorem E]{erw_psc2}): if $N$ is closed, $2$-connected and of dimension $d\ge6$, there exists a \emph{right-stable} metric $g_{\rst}\in\calR^+(N\setminus \inn(D))_h$ for some boundary metric $h\in\calR^+(S^{d-1})$. By definition of right-stability, the gluing map 
\[\mu_{g_{\rst}}\colon \calR^+(D^d)_{h} \too \calR^+(N),\ g\mapsto g_{\rst}\cup g\]
is a weak homotopy equivalence, in particular bijective on components. If $M$ is as above and $h\in \calR^+(S^{d-1})$ extends to a psc-metric on both $D^d$ and $M\setminus\inn(D)$, then  \pref{Theorem}{main:hitchin} implies that \emph{no} metric in $\calR^+(M\setminus \inn(D))_h$ is right-stable.
	\item \pref{Theorem}{main:hitchin} also holds if one replaces $D$ by any closed, codimension $0$ submanifold $W\subset M$ which admits a spin structure. This follows easily from the above-mentioned additivity theorem for the index-difference by considering a disc inside $W$.
\end{enumerate}
\end{rem}

\noindent In dimensions $d\equiv3(4)$, we have $\ko^{-d-1}\cong \bbZ$. Since the map $\inddiff$ is surjective on components by \cite[Theorem D]{berw}, there is an infinite family $(g_n)_{n\in\bbN}\subset\calR^+(D^d)_{h_\circ}$ of pairwise non-isotopic psc-metrics. An elementary argument shows that these metrics become concordant\footnote{Two psc-metrics $g_0$ and $g_1$ on $M$ are called \emph{concordant} if $\calR^+(M\times[0,1])_{g_0\amalg g_1}\not=\emptyset$, i.e. if they can be connected by a psc-metric $G$ on the cylinder. If $M$ itself has boundary $\partial M$, we require the concordance to be cylindrical in a neighbourhood of $\partial M\times [0,1]$.} after extending them arbitrarily to a totally nonspin manifold. We present this proof in \pref{Section}{sec:concordance}.

\subsection*{Acknowledgements} I would like to acknowledge the Oberwolfach Research Institute for Mathematics for its hospitality during the workshop \enquote{Analysis, Geometry and Topology of Positive scalar curvature metrics}, where this project was initiated. I thank Thomas Schick and Johannes Ebert for an inspiring discussion and Bernhard Hanke as well as Simone Cecchini and Jannes Bantje for comments on a draft of this paper.

\section{Preliminaries}
\subsection{Spaces of metrics of positive scalar curvature on a manifold with boundary}\label{sec:prelims-psc}
For a closed manifold $M$ let $\calR^+(M)$ denote the space of all metrics of positive scalar curvature. If $M$ has non-empty boundary $\partial M$ let $c\colon [0,1]\times\partial M \embeds M$ be a collar embedding such that $c(\{0\}\times \partial M)= \partial M\subset M$ and $c|_{\{0\}\times\partial M}$ is the canonical projection. We define $\calR^+(M)$ to consist of all metrics $g$ on $M$ such that there exists a psc-metric $h$ on $\partial M$ and an $\epsilon>0$ with
\[c^*g|_{[0,\epsilon] \times\partial M} = \dt^2 + h,\]
i.e. we require metrics on $M$ to be \emph{cylindrical in a neighbourhood of the boundary}. One nice feature about having cylindrical boundary is that the restriction map
\[\res\colon\calR^+(M)\too\calR(\partial M)\] 
actually lands in the subspace $\calR^+(\partial M)$ of psc-metrics. For $h\in\calR^+(\partial M)$ we define $\calR^+(M)_h\coloneqq \res^{-1}(h)$. By \cite[Theorem 1.6]{ebertfrenck} the map $\res$ is a Serre-fibration and hence for $h\in\calR^+(\partial M)$ and $g\in\calR^+(M)_h$ we have a long exact sequence of homotopy groups
\begin{equation}\label{eq:fibration}
	\dots\to\pi_k(\calR^+(M),g)\to\pi_k(\calR^+(\partial M),h)\overset{\partial}\to\pi_{k-1}(\calR^+(M)_h,g)\to\dots
\end{equation}
To prove \pref{Theorem}{main:boundary} (i) we need to get a more explicit description of the boundary map in this exact sequence. For our purposes it suffices to identify the boundary map on the image of the orbit of pullback-action of the topological group\footnote{This also carries the weak $C^\infty$-topology. Note that with this topology, the pullback map $\diff_\partial(M)\to\calR^+(M)$ is continuous.} $\diff_\partial(M)$ of self-diffeomorphisms of $M$. Let $(f_t)_{t\in [0,1]^k}$ be a continuous family of diffeomorphisms of $\partial M$ such that $f_t=\id$ for $t$ near $\partial [0,1]^k$. For $s\in[0,1]^{k-1}$ we define $F_s\colon M\to M$ by
\begin{equation}\label{eq:fiber-transport}
	F_s(x) = \begin{cases}
	x & \text{ if } x\not\in\im c\\
	c(r, f_{s,(1-r)}(y)) &\text{ if } x=c(r,y)
\end{cases}
\end{equation}
Note that by our assumption on the family $(f_t)$ we have $F_s=\id$ for $s$ near $\partial [0,1]^{k-1}$ and $c(r, f_{s,r}(y)) = c(r,y)$ for $r$ near $0$ or $1$. Therefore, $F_s$ is a smooth diffeomorphism and restricts to the identity in a neighbourhood of the boundary.

\begin{lem}\label{lem:boundary-map}
	Let $h\in\calR^+(\partial M)$ and $g\in \calR^+(M)_h$ be base points. Then \[\partial\bigr((f_t^*h)_{t\in [0,1]^k}\bigr)=(F_s^*g)_{s\in[0,1]^{k-1}}.\]
\end{lem}

\begin{proof}
	Consider the following lifting problem: 
\begin{center}
\begin{tikzpicture}
	\node (0) at (0,1.2) {$J^{k-1}$};
	\node (1) at (0,0) {$[0,1]^{k}$};
	\node (2) at (3,1.2) {$\calR^+(M)$};
	\node (3) at (3,0) {$\calR^+(\partial M)$};
	
	\draw[left hook->] (0) to (1);
	\draw[->] (0) to (2);
	\draw[->] (1) to node[above, sloped]{$\alpha$} (3);
	\draw[->] (2) to (3);
	\draw[->, dashed] (1) to node[above, sloped]{$\beta$} (2);
\end{tikzpicture}
\end{center}
where $J^{k-1}\coloneqq [0,1]^{k-1}\times\{0\})\cup (\partial [0,1]^{k-1}\times[0,1])$ and $\alpha(t) = f_t^*h$. Since the boundary map in (\ref{eq:fibration}) is determined by $\partial(\alpha) = \beta|_{[0,1]^{k-1}\times\{1\}}$, we need to specify a lift $\beta$. Note that in our case $J^{k-1}\subset \partial[0,1]^k$ maps to the base point metric $g$. We define a family $\tilde F_{s,\ell}$ of diffeomorphisms of $M$ for $s\in[0,1]^{k-1}$ and $\ell\in[0,1]$ as follows:
\[\tilde F_{s,\ell}(x)\coloneqq\begin{cases}
	x &\text{ if } x\not\in\im c\\
	c(r, f_{s,(1-r)\ell}(y)) &\text{ if } x=c(r,y)	
\end{cases}.\]
This satisfies: 
\begin{itemize}
	\item[-] $\tilde F_{s,\ell} = \id$ for $(s,\ell)\in J^{k-1}$.
	\item[-] $\tilde F_{s,\ell}(c(0,y)) = c(0, f_{s,\ell}(y)) = f_{s,\ell}(y)$.
	\item[-] $\tilde F_{s,1}(c(r,y)) = c(r,f_{s,(1-r)}(y))$.
\end{itemize}
In particular, one possible choice of $\beta$ is given by $\tilde F_{s,\ell}^*g$ and the proof is finished by the observation that $\tilde F_{s,1} = F_s$ for $F_s$ the diffeomorphism defined in (\ref{eq:fiber-transport}).
\end{proof}

\subsection{Recollection of existence results for positive scalar curvature}\label{sec:psc-recollection}

Before we get to positive scalar curvature, let us recall tangential structures. Let $\bo(n)$ denote the classifying space for rank $n$-vector bundles. Extending an orthonormal matrix by $1$ in the lower right corner induces a stabilisation map $\bo(n)\to\bo(n+1)$. A \emph{tangential structure} is defined to be a fibration $\theta\colon B\to \bo(n)$. Given a tangential structure $\theta$ a \emph{$\theta$-structure} on a $\rk n$-vector bundle $V\to X$ is a bundle map $\hat\ell\colon V\to \theta^*U_n$ for $U_n\to B\ort(n)$ the universal rank $n$ vector bundle. A (stabilised) $\theta$-structure on  a manifold $M$ of dimension $d\le n$ is a $\theta$-structure on $TM\oplus \underline\bbR^{n-d}$ and a (stabilised) $\theta$-manifold is a manifold together with a $\theta$-structure. For $d<n$ we denote the cobordism group of closed $\theta$-manifolds by $\Omega_d^\theta$.\footnote{This condition ensures that we have a notion of $\theta$-structure on cobordisms of $d$-manifolds.} 

\begin{example}
Consider the fibrations $\theta_{\Spin,\pi}\colon \bspin(n)\times B\pi\to \bo(n)$ and $\theta_{\so,\pi}\colon \bso(n)\times B\pi\to \bo(n)$. The corresponding cobordism groups are simply the singular cobordism groups $\Omega_d^\Spin(B\pi) = \Omega_d^{\theta_{\Spin,\pi}}$ and $\Omega_d(B\pi) = \Omega_d^{\theta_{\so,\pi}}$. If $M_0$ and $M_1$ represent the same class in $\Omega_d^\Spin(B\pi)$ (resp. $\Omega_n^{\theta_{\so,\pi}}$), we will say that $M_0$ and $M_1$ are \emph{$\Spin\times B\pi$-cobordant} (resp. \emph{$\so\times B\pi$-cobordant}). In the special case that $M_0=\emptyset$, we say $M_1$ is \emph{$\Spin\times B\pi$-nullbordant} (resp. \emph{$\so\times B\pi$-nullbordant}).
\end{example}

\noindent Recall the following useful lemma, stating that doubles are $\theta$-cobordant to cylinders.

\begin{lem}[{\cite[Proposition 3.25]{actionofmcg}}]\label{lem:double}
		Let $W^d\colon M_0 \leadsto M_1$ be a $\theta$-cobordism. Then there exists a $\theta$-structure on $W^\op\colon M_1\leadsto M_0$ such that $W\cup W^\op$ is $\theta$-cobordant to $M_0\times [0,1]$ relative to $M_0\times\{0,1\}$.
\end{lem}

\noindent For positive scalar curvature questions, the relevant tangential structure one needs to consider is usually the tangential $2$-type of the underlying manifold.

\begin{dfn}
	Let $M$ be a connected manifold of dimension $d<n$ and let $\tau\colon M\to\bo(n)$ be the stabilised classifying map of the tangent bundle. For $k\ge1$ we define the \emph{stabilised tangential $k$-type of $M$} as the $k$-th stage of the Moore-Postnikov tower for the map $\tau$.
\end{dfn}

\begin{example}\label{ex:2-types}
	\begin{enumerate}
		\item If $M$ is $\Spin$ and $n>d$, the stabilised tangential $2$-type of $M$ is given by $\bspin(n)\times B\pi_1(M)$.
		\item If $n>d$ and $M$ is orientable, totally nonspin, the stabilised tangential $2$-type of $M$ is given by $\bso(n)\times B\pi_1(M)$.
%		\item If $M$ is non-orientable, the tangential $2$-type is of $M$ can be described as follows: Let $\pi\coloneqq \pi_1(M)$ and let $u\colon M\to B\pi$ be a classifying map for the universal cover of $M$. Since $\tilde M$ is orientable, the first Stiefel--Whitney class of $M$ is given by $u^*\beta$ for some $\beta\in H^1(B\pi;\bbZ/2)$ and we interpret $\beta$ as a map $B\pi\to B\bbZ/2$. The tangential $2$-type of $M$ is then given as the pullback of the following diagram
%\begin{center}
%\begin{tikzpicture}
%	\node (0) at (0,1.2) {$B$};
%	\node (1) at (0,0) {$B\ort(d+1)$};
%	\node (2) at (3,1.2) {$B\pi$};
%	\node (3) at (3,0) {$B\bbZ/2$};
%	
%	\draw[->] (0) to node[left]{$\theta$} (1);
%	\draw[->] (0) to (2);
%	\draw[->] (1) to node[above]{$w_1(U_{d+1})$} (3);
%	\draw[->] (2) to node[right]{$\beta$} (3);
%\end{tikzpicture}
%\end{center}
%where $w_1(U_{d+1})$ is the universal first Stiefel--Whitney class. Note that we have maps $B\Spin(d+1)\to \bso(d+1)\to B$ and hence, there are induced maps 
%\[\Omega_{d+1}^{\Spin}\too\Omega_{d+1} \too\Omega_{d+1}^\theta.\]
	\end{enumerate}
\end{example}

\noindent The reason, why the tangential $2$-type shows up in the study of positive scalar curvature metrics, originates from the famous Gromov--Lawson--Schoen--Yau surgery theorem \cite{gl80a, schoenyau_classical}. It states that the existence of a psc-metric is invariant under surgeries of codimension at least $3$ on the underlying manifold. We have the following, cobordism-theoretic translation of the surgery theorem.

\begin{lem}[{\cite[Theorem A]{gl80a}}, {\cite[Theorem 1.5]{ebertfrenck}}]\label{lem:gajer}
	Let $M$ be $\Spin$ (resp. totally nonspin) and $\Spin\times B\pi_1(M)$-cobordant (resp. $\so\times B\pi_1(M)$-cobordant) to some manifold $N$ of positive scalar curvature. Then $M$ admits positive scalar curvature. If $N$ and $M$ have equal boundary and the cobordism is constant on the boundary the same result holds and the metrics on $M$ and $N$ can be chosen to restrict to the same metric on the boundary.
\end{lem}

\noindent The version for manifolds with boundary can be derived from \cite[Theorem 1.6 and proof of Theorem 1.5]{ebertfrenck}. We also need an existence result for psc-metrics on cobordisms. First, let us recall the notion of right-stable metrics. Let $W\colon M_0\leadsto M_1$ be a manifold with boundary $\partial W=M_0\amalg M_1$ and $h_i\in\calR^+(M_i)$. Then a metric $g\in\calR^+(W)_{h_0,h_1}$ is called \emph{right-stable} if for every cobordism $V\colon M_1\leadsto M_2$, the map 
\[\mu_g\colon \calR^+(V)_{h_1,h_2}\to \calR^+(W\cup V)_{h_0,h_2},\ \ G\mapsto g\cup G\]
is a weak homotopy equivalence. On the $\pi_0$-level, this implies that every psc-metric on $M$ is homotopic to one which restricts to $g$ on $W$. Ebert--Randal-Williams have proven the following result.
\begin{thm}[{\cite[Theorem E]{erw_psc2}}]\label{thm:right-stable}
	If $\dim W\ge 6$ and $M_1\embeds W$ is $2$-connected, then for any $h_0\in\calR^+(M_0)$, there exists an $h_1\in\calR^+(M_1)$ and a right-stable metric $g\in\calR^+(W)_{h_0,h_1}$.
\end{thm}

\begin{rem}\label{rem:chernysh}
\begin{enumerate}
	\item There is no restriction on $M_0$ except for admitting positive scalar curvature. In particular, if $M_0$ is empty, \pref{Theorem}{thm:right-stable} implies the existence of right-stable metrics on nullcobordisms if $\partial W\embeds W$ is $2$-connected. Therefore, we have the following implication: If $W\subset M$ is a closed, codimension $0$ submanifold such that $\partial W\embeds W$ is $2$-connected, then there exists a psc-metric $g\in\calR^+(W)_h$ such that the gluing map
	\[\mu_g\colon \calR^+(M\setminus \inn(W))_{h}\to \calR^+(M),\ \ G\mapsto g\cup G\]
	is a weak homotopy equivalence.
	\item A parametrised version of the surgery theorem has been proven by Chernysh \cite{chernysh} (see also \cite{ebertfrenck} and \cite{kordass}). In its general form, it states that for $k\ge3$ the metric $g_N + g_\tor$ is a right-stable metric on $N\times D^k$, $g_\tor$ a torpedo metric\footnote{A torpedo-metric is an $O(k)$-invariant psc-metric on $D^k$ that is cylindrical near the boundary and restricts to the round metric on the boundary (cf. \cite[Definition 2.9]{ebertfrenck}).} and $g_N$ any metric on $N$ such that $g_N+g_\tor$ has positive scalar curvature. In particular, the torpedo metric on $D^k$ itself is right-stable. \pref{Theorem}{thm:right-stable} can be seen as an improved parametrised surgery theorem.
\end{enumerate}
\end{rem}

\subsection{The action of $\diff(M)$ on $\calR^+(M)$}

To prove \pref{Theorem}{main:boundary} (i) and \pref{Theorem}{main:hitchin}, we use a rigidity result for the pullback-action $\diff(M)\actson \calR^+(M)$. Let us recall the definition of the structured mapping class group.

\begin{dfn}
For $(M,\hat l)$ be a stabilised $\theta$-manifold, the \emph{$\theta$-mapping class group} denoted by $\Gamma^\theta(M,\hat l)$ is given by
\[\Gamma^\theta(M,\hat l)\coloneqq\left\{(f,L)\colon \begin{array}{l}f\colon M\congarrow M\text{ is a diffeomorphism}\\L\text{ is a homotopy of bundle maps } \hat l\circ df\leadsto \hat l \end{array}\right\}\Big/\sim\]
where the equivalence relation is given by homotopies of $f$ and $L$. 
\end{dfn}
\noindent There is a forgetful homomorphism 
\begin{equation}\label{eq:forgetful}
	\Phi\colon \Gamma^\theta(M, \hat l)\to\pi_0(\diff(M))
\end{equation} to the ordinary mapping class group, which in general is neither surjective nor injective. By composing $\Phi$ with the pullback action, we get an action $\Gamma^\theta(M,\hat l)\actson \pi_0(\calR^+(M))$. Furthermore, there is a homomorphism $T\colon\Gamma^\theta(M,\hat l)\to\Omega_{d+1}^\theta$ mapping $(f,L)$ to the \emph{mapping torus} $T_f$ with a $\theta$-structure determined by $L$. We have the following rigidity result about this pullback action. 

\begin{thm}[{\cite[Corollary 3.32]{actionofmcg}}]\label{thm:action-of-mcg}
	If $\theta$ is the stabilised tangential $2$-type of a closed manifold $M$ of dimension at least $6$, then the action $\Gamma^\theta(M,\hat l)\actson \pi_0(\calR^+(M))$ factors through $\Omega_{d+1}^\theta$ via the mapping torus construction. In particular, if $T_f$ is $\theta$-cobordant to a mapping torus of the identity, then $f^*g$ is homotopic to $g$ for all $g\in\calR^+(M)$.
\end{thm}

\noindent Note that a mapping torus of the identity need not be $\theta$-nullbordant, even though the underlying manifold $T_{\id}$ is simply given by $M\times S^1$: An example is already given by $S^1$ which is a mapping torus of $\id\colon \{\pt\}\to\{\pt\}$ but admits a spin structure which is not spin nullbordant. We need the following version of \pref{Theorem}{thm:action-of-mcg} for manifolds with boundary.

\begin{cor}\label{cor:action-of-mcg}
	Let $M$ and $\theta$ be as above and let $N\subset M$ be a closed, codimension $0$ submanifold such that the inclusion $\partial N\embeds N$ is $2$-connected. Let $(f,L)\in\Gamma^\theta(M, \hat l)$ such that $f|_N=\id$. Then there exists an $h\in\calR^+(\partial N)$ with the following property: If $T_{f}$ is $\theta$-cobordant to a mapping torus of the identity, then $(f|_{M\setminus N})^*g$ is homotopic to $g$ for all $g\in\calR^+(M\setminus N)_{h}$.
\end{cor}

\begin{proof}
By \pref{Theorem}{thm:right-stable} the assumption on $N$ implies that there exists an $h\in\calR^+(\partial N)$ and a right-stable metric $g_\rst\in\calR^+(N)_h$. The map $\calR^+(M\setminus N)_{h}\embeds\calR^+(M)$ given by gluing in $g_\rst$ is hence a homotopy equivalence. Since $f|_{N}=\id$, we have the following commutative diagram
\begin{center}
\begin{tikzpicture}
	\node (0) at (0,1.2) {$\calR^+(M\setminus N)_{h}$};
	\node (1) at (0,0) {$\calR^+(M\setminus N)_{h}$};
	\node (2) at (3,1.2) {$\calR^+(M)$};
	\node (3) at (3,0) {$\calR^+(M)$};
	
	\draw[->] (0) to node[left]{$(f|_{M\setminus N})^*$} (1);
	\draw[->] (0) to node[above]{$\simeq$} (2);
	\draw[->] (1) to node[above]{$\simeq$} (3);
	\draw[->] (2) to node[right]{$f^*$} (3);
\end{tikzpicture}
\end{center}
where the horizontal maps are equal homotopy equivalences. Therefore $f^*$ is homotopic to the identity if and only if $(f|_{M\setminus N})^*$ is homotopic to the identity. The rest follows from \pref{Theorem}{thm:action-of-mcg}.
\end{proof}

%\begin{rem}
%	The proof of \pref{Theorem}{thm:action-of-mcg} should in fact generalise to arbitrary boundaries and we expect that one can hence drop the assumption on $\partial M$ in the first half of \pref{Theorem}{main:boundary}\todo{think about this!}
%\end{rem}

\noindent We also need a detection result due to Hitchin \cite{hitchin_spinors} which has later been generalised by Crowley--Schick \cite{CS} and Crowley--Schick--Steimle \cite{CSS}. 
\begin{thm}[{\cite[Theorem 4.7]{hitchin_spinors}, \cite[Corollary 1.2]{CS}, \cite[Corollary 1.9]{CSS}}]\label{thm:css}
	Let $N$ be a compact spin manifold of dimension $d\ge6$, $h\in\calR^+(\partial N)$ and  $g\in\calR^+(N)_h$. Let furthermore $k\ge0$ such that $d+k+1\equiv1,2(8)$. Then the composition
	\[\pi_k(\diff_\partial(D^d))\to\pi_k(\diff_\partial(N))\to \pi_k(\calR^+(N)_h)\to \ko_{d+k+1}\cong\bbZ/2\]
	is split surjective. Here, the first map is given by extending a diffeomorphism on an embedded disk by the identity, second map is induced by the orbit map $f\mapsto f^*g$ and the last map is the index-difference. 
\end{thm}

\begin{rem} 
	Even though this result is only stated for closed manifolds in the papers \cite{hitchin_spinors, CS, CSS}, a version for manifolds with boundary is easily deduced.
\end{rem}

\noindent Next, we will give an explicit model for the homomorphisms
\[\lambda_{k,d}\colon \pi_k(\diff_\partial(D^d))\to \pi_{k-1}\diff_\partial(D^{d+1})\]
inducing the Gromoll-Filtration of homotopy spheres: For $n\ge1$ we fix once and for all $c_n\coloneqq \frac{1}{2\sqrt{n}}$ and note that $[0,c_n]^d\subset D^d$ for every $n\ge d$. Let $f_\sigma\colon[0,1]^k\to \diff_\partial(D^d)$ represent an element $\sigma\in\pi_k(\diff_\partial(D^d))$. After possibly applying a suitable compressing homotopy, we may assume that $f_\sigma(t)(x) = x$ for $(t,x)$ outside of $[0,c_{d+1}]^k\times [0,c_{d+1}]^d$, i.e. $f_\sigma$ is supported on the cube $[0,c_{d+1}]^{d+k}$. For $s\in [0,c_{d+1}]^{k-1}$, $x\in[0,c_{d+1}]^d$ and $l\in[0,c_{d+1}]$ we define
\[F_\sigma(s)(x,\ell)\coloneqq \bigl(\ell,f_\sigma(s,\ell)(x)\bigr).\]
By our assumption on $f_\sigma$ we see that $F_\sigma(s) = \id$ on $\partial([0,c_{d+1}]^{d+1})$ and by our choice of $c_{d+1}$ we can hence extend this by the identity to a map $F_\sigma\colon[0,1]^{k-1}\to \diff_\partial (D^{d+1})$ which is a representative $\lambda_{k,d}(\sigma)$.

 For an element $\sigma\in\pi_k(\diff_\partial(D^d))$ the successive composition of the maps $\lambda_{m,d+m}$ for $0\le m\le k$ yields an element $[f_\sigma]\in\pi_0(\diff_\partial(D^{d+k}))$ which defines a homotopy sphere $\Sigma$ by clutching along $f_\sigma$. The composition in \pref{Theorem}{thm:css} is equal to mapping $\sigma$ to the $\alpha$-invariant of $\Sigma$. We call homotopy spheres with non-vanishing $\alpha$-invariant \emph{Hitchin spheres}. 

\begin{lem}\label{lem:hitchin-boundary}
	Let $M^d$ be a manifold with boundary $\partial M$ and let $h\in\calR^+(\partial M)$ and $g\in \calR^+(M)_h$. Then there exist embeddings $D^{d-1}\embeds \partial M$ and $D^{d}\embeds M$ such that the following diagram commutes
\begin{center}
\begin{tikzpicture}
	\node (0) at (0,1.2) {$\pi_k(\diff_\partial(D^{d-1}))$};
	\node (1) at (0,0) {$\pi_{k-1}(\diff_\partial(D^d))$};
	\node (2) at (4,1.2)  {$\pi_k(\calR^+(\partial M),h)$};
	\node (3) at (4,0) {$\pi_{k-1}(\calR^+(M)_{h},g)$};
	
	\draw[->] (0) to node[left]{$-\lambda_{k,d-1}$} (1);
	\draw[->] (0) to node[above]{$\rho$} (2);
	\draw[->] (1) to node[above]{$\rho$} (3);
	\draw[->] (2) to node[right]{$\partial$} (3);
\end{tikzpicture}
\end{center}
where the horizontal maps $\rho$ are given by extending by the identity outside of the embedded disks and acting via pullback.
\end{lem}

\begin{proof}
	Let $\sigma\in\pi_k(\diff_\partial(D^{d-1}))$ and let $f_\sigma$ be a representative, for which we again may assume that $f_\sigma(t)(x) = x$ for $(t,x)\not\in[0,c_{d}]^k\times [0,c_{d}]^{d-1}$ (see above for the definition of $c_d$). Let $D^{d-1}\subset \partial X$ be some embedded disk and let $f_\sigma$ be extended by the identity to $\partial X$. By \pref{Lemma}{lem:boundary-map} the composition $\partial\circ \rho$ is given by pullback along the following family of diffeomorphisms
	\begin{equation}\label{eq:above}
		F(s)(x) = \begin{cases}
			x & \text{ if } x\not\in\im c\\
			c(r, f_\sigma(s,(1-r))(y)) &\text{ if } x=c(r,y)
		\end{cases}
	\end{equation}
	where $c\colon [0,1]\times  \partial X\embeds X$ is the collar embedding. Applying a compressing and shifting homotopy, we may assume that $F(s)(c(r,y)) = c(r, f_\sigma(s,(c_d-r))(y))$ By our assumption on $f_\sigma$, $F$ is concentrated in an embedded cube $[0,c_d]^d\subset S^{d-1}\times[0,1]$. We choose an embedded disk $\iota_D\colon D^d\subset \partial X\times[0,1]$ that extends the inclusion of $[0,c_d]^{d}\subset \partial X\times[0,1]$. This is depicted in \pref{Figure}{fig:collar}.
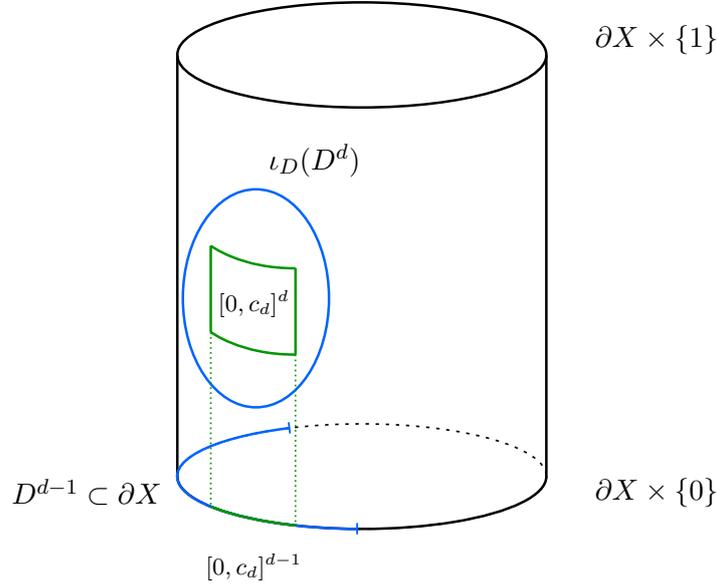
\begin{figure}[ht]
\begin{tikzpicture}
	\node at (0,0) {\includegraphics[width=0.5\textwidth]{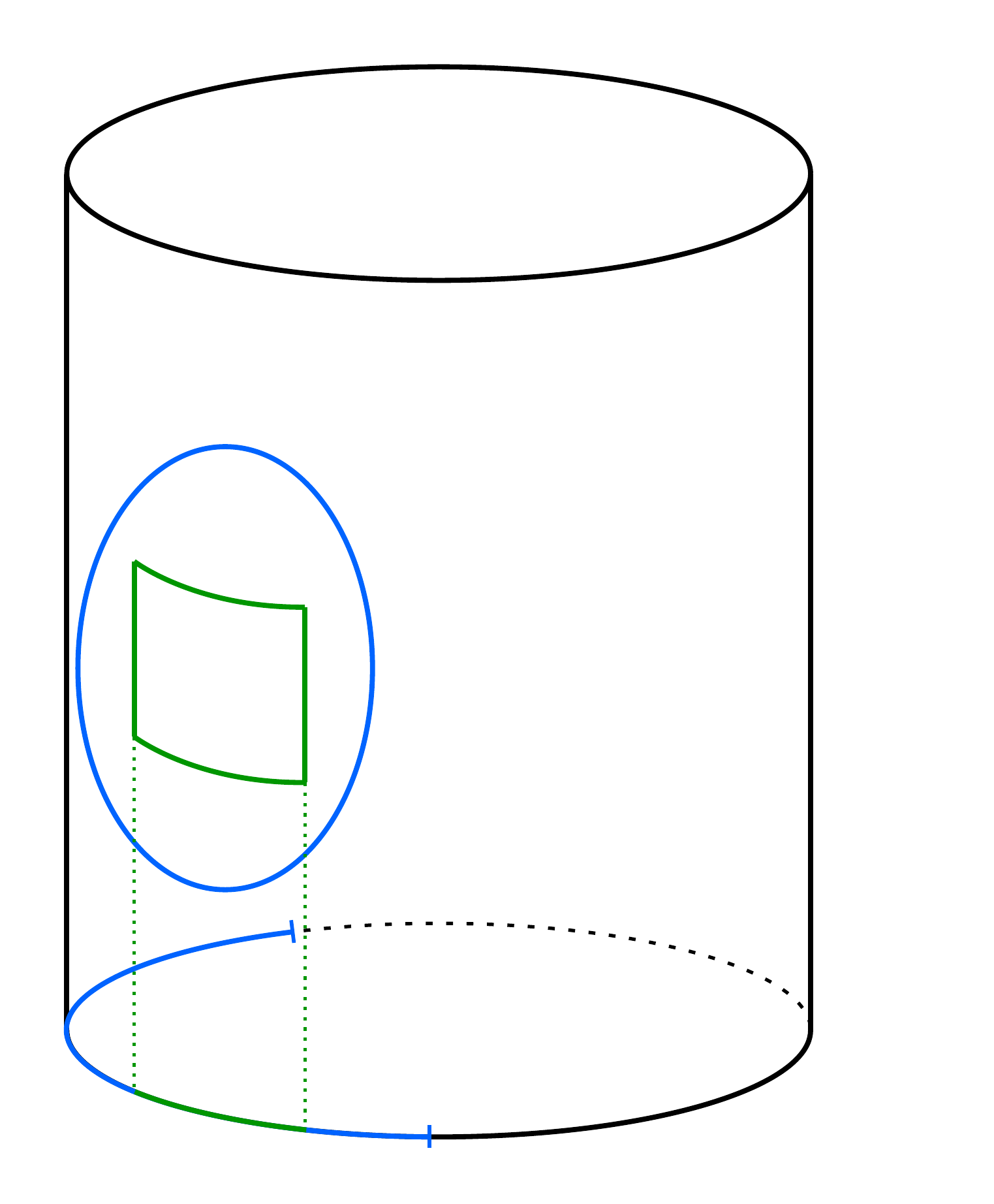}};
	\node at (-1.8,-0.5) {\footnotesize$[0,c_d]^d$};	
	\node at (-1.8,-4) {\footnotesize$[0,c_d]^{d-1}$};
	\node at (-1,1.4) {$\iota_D(D^{d})$};
	\node at (3.5,3) {$\partial X\times\{1\}$};
	\node at (3.5,-3) {$\partial X\times\{0\}$};
	\node at (-4,-3) {$D^{d-1}\subset \partial X$};
\end{tikzpicture}
\caption{The collar $\partial X\times [0,1]$ of $X$. }\label{fig:collar}
\end{figure}\\
	For $s\in[0,1]^{k-1}$ and $(y,\ell)\in D^{d}$, a representative of $-\lambda_{k,d-1} (\sigma) = \lambda_{k,d-1}(-\sigma)$ is given by
	\begin{align*}
		F_{-\sigma}(s)(y,\ell) &= (\ell,f_{-\sigma}(s,\ell)(y))\\
			& = (\ell,f_\sigma(s,c_d-\ell)(y))
	\end{align*}
	since $f_\sigma$ was supported on $[0,c_d]^{k}\times[0,c_d]^{d-1}$. After applying an appropriate shifting and rescaling homotopy that stretches the embedded cube to the collar, we see that $\iota_D\circ(-\lambda_{k,d-1})(\sigma)$ is represented by $F$ in (\ref{eq:above}) which proves the lemma.
\end{proof}

\noindent Let us close this section with the following well known proposition which is the key observation enabling us to employ \pref{Theorem}{thm:action-of-mcg} and \pref{Corollary}{cor:action-of-mcg}. 

\begin{prop}\label{prop:nullbordant}
	Every homotopy sphere is oriented nullbordant.
\end{prop}

\begin{proof}
	Oriented bordism is classified by Pontryagin- and Stiefel--Whitney numbers. For a homotopy sphere all characteristic classes have to vanish except for possibly the top ones. However, the top Pontryagin number is a nonzero multiple of the signature and the top Stiefel--Whitney number is the mod $2$ reduction of the Euler characteristic. Both of these vanish for homotopy spheres.
\end{proof}

\section{Proofs of the main results}\label{sec:proofs}

\noindent The proofs for \pref{Theorem}{main:hitchin} and the first half of \pref{Theorem}{main:boundary} (i) both rely on the following lemma.

\begin{lem}\label{lem:rigidity}
	Let $d\ge6$ and let $f_\Sigma \colon (D^d,\partial D^d) \to (D^d, \partial D^d)$ be a diffeomorphism as in \pref{Theorem}{thm:css} and let $M^d$ be a compact, oriented, totally nonspin-manifold with spin boundary $\partial M$ which is $\Spin\times B\pi_1(\partial M)$-nullbordant or empty. Let furthermore $D\subset M$ be an embedded disk away from the boundary and let $f\coloneqq f_\Sigma\cup\id$. Then, there exists an $h\in\calR^+(\partial M)$ such that for all $g\in \calR^+(M)_h$, the metric $f^*g$ is homotopic to $g$.
\end{lem}

\begin{rem}
If $\partial M=\emptyset$, then the conclusion of the lemma holds for all $g\in\calR^+(M)$.
\end{rem}

\begin{proof}
	By the assumption on $\partial M$, there is a nullbordism $W$ of $\partial M$ such that the inclusion $\partial M\embeds W$ is $2$-connected. This ensures that we are in a situation where \pref{Theorem}{thm:action-of-mcg} (if $\partial M=\emptyset$) and \pref{Corollary}{cor:action-of-mcg} (if $\partial M\not=\emptyset$) are applicable.

	If $M$ is orientable with fundamental group $\pi$ then its stabilised tangential $2$-type is given by the projection $\bso(d+1)\times B\pi\to \bso(d+1)\to\bo(d+1)$. A $\theta$-structure in this case is a bundle map $TM\to \theta^*U_{d+1} \cong \theta_{SO}^*U_{d+1} \times B\pi$ for $\theta_{SO}\colon \bso(d+1)\to B\ort(d+1)$ and hence, it is equivalently described by an orientation and a map $\gamma\colon M\to B\pi$. Therefore, an element in $\Gamma^\theta(M,\gamma)$ is given by a pair $(f,L)$ of an orientation preserving diffeomorphism $f$ and a homotopy $L$ of the maps $\gamma$ and $\gamma\circ f$.
	
	Note that the diffeomorphism $f_\Sigma$ is supported in the disk $D\subset M$ and hence fixes all points $p\in M\setminus D$. Furthermore, the induced map $f_*\colon\pi_1(M,p)\to \pi_1(M,p)$ is the identity since any loop is homotopic to one, which does not intersect $D$. Therefore, the maps $\gamma$ and  $\gamma\circ f$ induce equal maps on fundamental groups. The space $B\pi=K(G,1)$ is an Eilenberg-Maclane space and maps from $CW$-complexes into Eilenberg-Maclane spaces are determined uniquely up to homotopy by the induced map on $\pi_1$ (see e.g. \cite[Proposition 1B.9]{hatcher_at}). Therefore, we deduce that the maps $\gamma$ and $\gamma\circ f$ are homotopic. This implies that $f$ is in the image of the forgetful map $\Phi\colon\Gamma^\theta(M,\gamma)\to\pi_0\diff(M)$ from (\ref{eq:forgetful}). The mapping torus of $f$ is $\so\times B\pi_1(M)$-cobordant to $[M\times S^1\amalg \Sigma, \beta]$ in $\Omega_{d+1}(B\pi)$, where $\Sigma$ is the Hitchin sphere associated to $f_\Sigma$ and $\beta$ is constant on $\Sigma$. By \pref{Proposition}{prop:nullbordant}, this is $\so\times B\pi_1(M)$-cobordant to a mapping torus of the identity. We conclude from \pref{Theorem}{thm:action-of-mcg} and \pref{Corollary}{cor:action-of-mcg} that there exists a boundary metric $h\in\calR^+(\partial M)$ such that $f^*\colon \calR^+(M)_h\to\calR^+(M)_h$ is homotopic to the identity and in particular $f^*g\sim g$.	
%	In the case that $M$ is not orientable, its tangential $2$-type is given by the following pullback from \pref{Example}{ex:2-types}:
%\begin{center}
%\begin{tikzpicture}
%	\node (0) at (0,1.2) {$B$};
%	\node (1) at (0,0) {$B\ort(d+1)$};
%	\node (2) at (3,1.2) {$B\pi$};
%	\node (3) at (3,0) {$B\bbZ/2$};
%	
%	\draw[->] (0) to node[left]{$\theta$} (1);
%	\draw[->] (0) to (2);
%	\draw[->] (1) to node[above]{$w_1(U_{d+1})$} (3);
%	\draw[->] (2) to node[right]{$\beta$} (3);
%\end{tikzpicture}
%\end{center}
%	Therefore a $\theta$-structure on a manifold $X$ is determined by a map $\gamma\colon X\to B\pi$ such that $\beta\circ\gamma$ agrees with the wirst Stiefel--Whitney class of $X$. Again, diffeomorphism $f$ acts trivial on the fundamental group and hence acts trivially on the $\theta$-Structure of $M$. Therefore, we can again apply \pref{Theorem}{thm:action-of-mcg} and \pref{Corollary}{cor:action-of-mcg} and it suffices to consider the bordism class of the mapping torus $T_f$ in $\Omega_{d+1}^\theta$. Again, $T_f$ is bordant to the disjoint union of a mapping torus of the identity and $\Sigma$. Since $\Sigma$ is orientable and oriented nullbordant, it is in the kernel of the map \[\Omega_d^{\Spin}\too\Omega_d \too\Omega_d^\theta\] and we again conclude that $f^*$ is homotopic to the identity.
\end{proof}

\begin{proof}[Proof of \pref{Theorem}{main:hitchin}]
Let $g\in\calR^+(M\setminus \inn(D))_{h}$ and $g_0\in\calR^+(D^d)_{h}$ be arbitrary. By \pref{Theorem}{thm:css} there exists a diffeomorphism $f_\Sigma\colon (D^d, \partial D^d)\to(D^d, \partial D^d)$ such that $\inddiff_{g_0}(f_\Sigma^*g_0) \not=0$ and hence $g_1\coloneqq f_\Sigma^*g_0$ is not homotopic to $g_0$. By \pref{Lemma}{lem:rigidity} we have
\[g_1\cup g = (f_\Sigma^*g_0)\cup g = (f_\Sigma\cup \id) (g_0\cup g) \overset{\ref{lem:rigidity}}\sim g_0\cup g\]
which concludes the proof.
\end{proof}

\begin{proof}[Proof of \pref{Theorem}{main:boundary} (i)]
	Let $D^{d-1}\subset \partial M$ be an embedded disk and let $h$ be as in \pref{Lemma}{lem:rigidity}. By \pref{Theorem}{thm:css} there exists a family $(f^D_t)_{t\in S^1}$ of self-diffeomorphisms of $(D^{d-1},\partial D^{d-1})$ such that for $f_t\coloneqq f_t^D\cup \id$ we have $\inddiff_h([f_t^*h])\not=0$ and hence $[f_t^*h]\not=0\in\pi_1(\calR^+(\partial M), h)$ . Consider the long exact sequence of homotopy groups associated to the fibration $\res\colon\calR^+(M)\to\calR^+(\partial M)$:
	\begin{align*}
		\dots\too\pi_1(\calR^+(M))\overset{\res}\too\pi_1(\calR^+(\partial M))&\overset{\partial}\too\pi_{0}(\calR^+(M)_{h_\circ})\too\dots\\	
		[f_t^*h]\qquad&\ \mapsto \qquad\partial[(f_t)^*g]
	\end{align*}
	Now, \pref{Lemma}{lem:hitchin-boundary} states that $\partial([f_t^*h])$ is given by $F^*g$ for $F$ is an extension of $-\lambda_{1,d-1}(f_t)$ by the identity. Therefore, $F$ is a diffeomorphism which clutches a Hitchin-sphere and \pref{Lemma}{lem:rigidity} implies that $F^*g$ and $g$ are homotopic. Hence $[f_t^*h]$ has to be in the image of the restriction map $\pi_1(\calR^+(M))\to\pi_1(\calR^+(\partial M))$.
\end{proof}

\begin{rem}
	It is not possible to use the rigidity result on the pullback action obtained in \cite[Theorem E]{erw_psc3}, as it requires the inclusion $\partial M\embeds M$ to be $2$-connected. Our proof crucially relies on $M$ and $\partial M$ having different tangential $2$-types. 
%	An extension of \cite[Theorem E]{erw_psc3} as mentioned in \cite[Remark 1.3.1]{erw_psc3} would deliver a version of \pref{Theorem}{main:boundary} (i) for higher homotopy groups.
\end{rem}

\begin{proof}[Proof of \pref{Theorem}{main:boundary} (ii)]
	Let $x\in\ko^{-d}(\pt)$, $h\in\calR^+(\partial M)$ and $g\in\calR^+(M)_h$. Since the $\alpha$-invariant $\Omega_d^\Spin\to\ko^{-d}(\pt)$ is surjective, there exists a closed spin manifold $X$ with $\alpha(X) = x$ and we may without loss of generality assume that $X$ is $2$-connected. By \pref{Theorem}{thm:right-stable} we conclude that there exists a (right-stable) metric $g_x\in\calR^+((\partial M\times[0,1])\#X)_{h,h_x}$ for some boundary metric $h_x\in\calR^+(\partial M)$. Note that $\inddiff_h([h_x]) = -x$ which follows from the definition of $\inddiff$ and the spin bordism invariance of the index of the spin Dirac operator. Gluing $g_x$ onto $g$, we obtain a psc-metric on $M\# X$ that restricts to $h_x$ on the boundary. Now, $X$ is oriented cobordant to a closed, oriented, totally nonspin manifold $X^\orient$, for example $X^\orient$ can be chosen to be $X\#(\cp 2\times S^{d-4})$. Note that $X^\orient$ which admits a metric of positive scalar curvature by \cite[Corollary C]{gl80a}. Hence there is a psc-metric on $M\#X\#(X^\orient)^\op$ restricting to $h_x$ on the boundary by \pref{Lemma}{lem:gajer}. But $M\#X\#(X^\orient)^\op$ is totally nonspin and $\so\times B\pi_1(M)$-cobordant to $M$ relative to the boundary. Therefore, there is a psc-metric $g_{M,x}$ on $M$ restricting to $h_x$ on the boundary which finishes the proof.
\end{proof}

\section{Concordance classes of psc-metrics}\label{sec:concordance}

\noindent Having proven \pref{Theorem}{main:hitchin}, it is natural to ask if all psc-metric on the disc become isotopic after extending them to a totally nonspin manifold. It is well-known that for $d\ge7$, $d\equiv 3(4)$ there is an infinite family $(g_n)_{n\in\bbN}\subset\calR^+(D^d)_{h_\circ}$ of pairwise non-isotopic and even non-concordant metrics\footnote{Recall that two psc-metrics $g_0$ and $g_1$ on $M$ are called \emph{concordant} if $\calR^+(M\times[0,1])_{g_0\amalg g_1}\not=\emptyset$, i.e. if they can be connected by a psc-metric $G$ on the cylinder. If $M$ itself has boundary $\partial M$, we require the concordance to be cylindrical in a neighbourhood of $\partial M\times [0,1]$. Note that isotopic metrics are concordant, see for example by \cite[Lemma 2.5]{ebertfrenck}}. Such a family can be constructed as follows:

Let $\beta$ be a $(d+1)$-dimensional, $2$-connected (spin) manifold, with $\alpha(\beta)=1\in\ko^{-d-1}(\pt)\cong\bbZ$. Let $\beta_n\coloneqq\beta^{\#n}$ denote the $n$-fold connected sum of $\beta$ with itself and let $W_n\coloneqq S^{d}\times [0,1]\#\beta_n$. Since $W_n$ is $2$-connected, it admits a metric $G_n'$ of positive scalar curvature, which is of product type in a neighbourhood of the boundary and restricts to the metric $g_\tor\cup g_\tor^\op$ on the lefthand boundary by \pref{Theorem}{thm:right-stable}. Furthermore, we may assume that the metric $G_n'$ restricted to the cylinder over the lower hemisphere $D_{-}^d\subset D_{-}^d\cup D^d_+ =  S^d$ is given by the cylinder over a torpedo-metric by Chernysh's parametrised surgery theorem (cf. \pref{Remark}{rem:chernysh} (ii)). Restricting to the complement of $\inn(D_{-}^d)\times[0,1]$ yields a psc-metric $G_n$ on $W_n\setminus (\inn(D_{-}^d)\times[0,1])$. The metric $G_n$ restricted to the boundary is given as follows: On the left-hand part of the boundary it is a torpedo-metric $g_\tor$, on the bottom part of the boundary it is equal to $h_\circ+\dt^2$ for $h_\circ$ the round metric and on the right-hand boundary it restricts to some metrics $g_n$, see \pref{Figure}{fig:boundary}. 

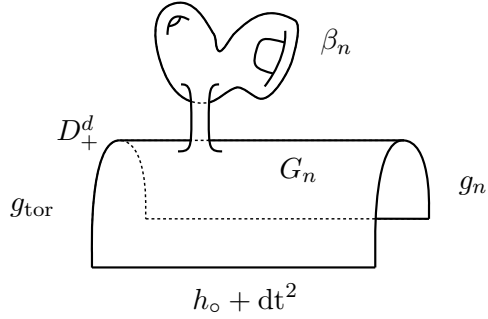
\begin{figure}[ht]
\begin{tikzpicture}
	\node at (0,0) {\includegraphics[width=0.4\textwidth]{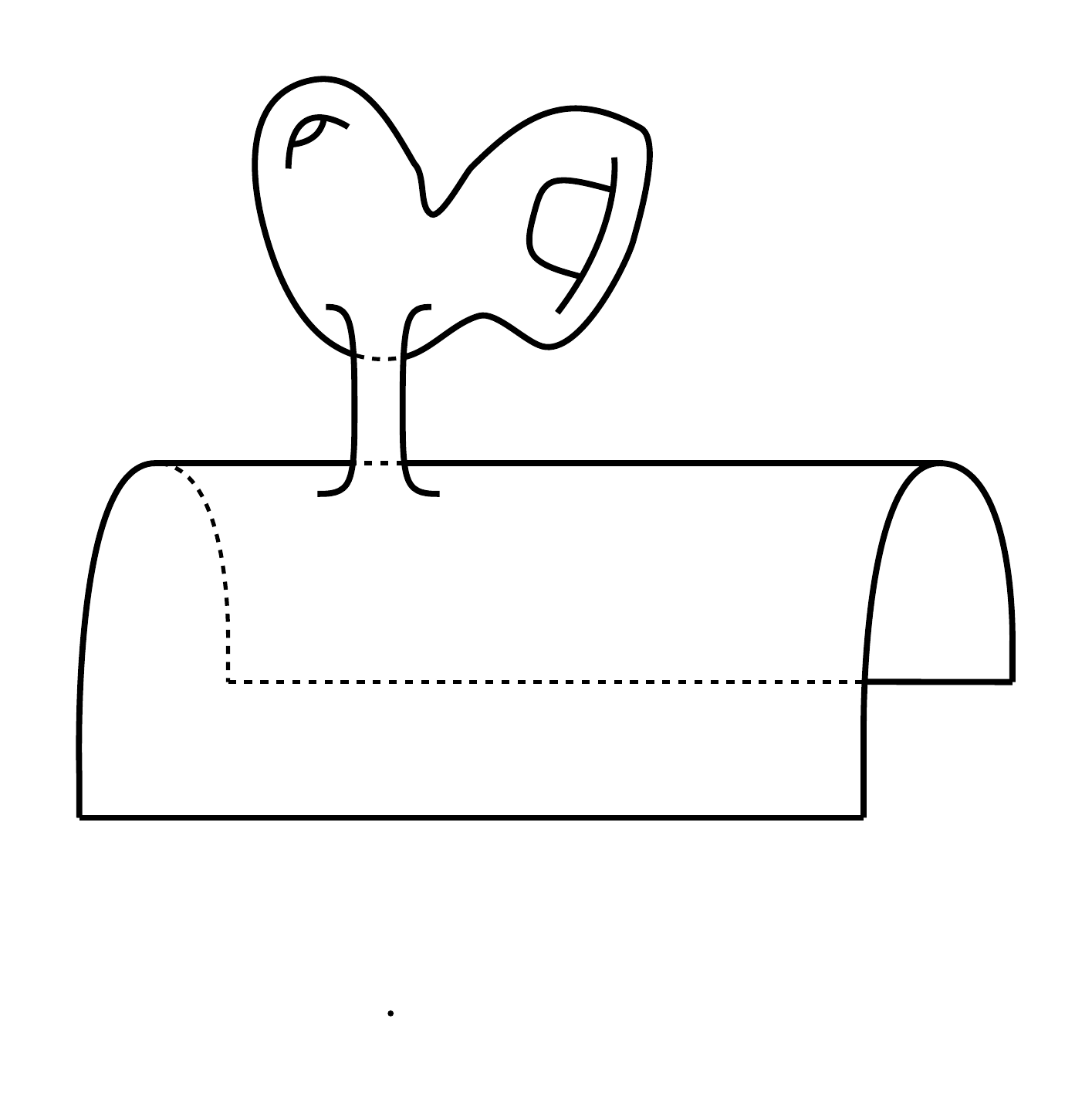}};
	\node at (1,1.7) {$\beta_n$};
	\node at (-2.4,0.5) {$D^d_+$};
	\node at (0.5,0) {$G_n$};
	\node at (-3,-0.5) {$g_\tor$};
	\node at (-0.2,-1.7) {$h_\circ + \dt^2$};
	\node at (2.8,-0.2) {$g_n$};
\end{tikzpicture}
\caption{The manifold $W_n\setminus (D_{-}^d\times [0,1])$ with the psc-metric $G_n$.}\label{fig:boundary}
\end{figure}

\noindent Again, by the spin bordism invariance of the index, we have $\inddiff_{g_\tor}(g_n) = \alpha(\beta_n) = n$, hence the metrics $(g_n)_{n\in\bbN}$ are pairwise not concordant.  We arrive at the following first instance of the question mentioned above.

\begin{question}\label{q:isotopic}
	Do the metrics $(g_n)_{n\in\bbN}$ become isotopic after extending them arbitrarily to a totally nonspin manifold?
\end{question}

\noindent Note that the proof of \pref{Theorem}{main:hitchin} relies crucially on the observation that Hitchin spheres are oriented nullbordant (see \pref{Proposition}{prop:nullbordant}) which then enables the use of \pref{Theorem}{thm:action-of-mcg} and \pref{Corollary}{cor:action-of-mcg}. The manifold $\beta$ however is not oriented nullbordant, since the $\alpha$-invariant of a spin manifold in these dimensions is equal to a multiple of its $\ahat$-genus, which is a Pontryagin number, hence it is not possible to use these rigidity results.\footnote{Note that the metrics $g_n$ are not in the image of the orbit map $\pi_0(\diff_\partial(D^d))\to\pi_0(\calR^+(D^d)_{h_\circ}$ by \cite[Theorem B]{actionofmcg}. This problem could however be solved by either directly utilising the cobordism result obtained in \cite{actionofmcg} or by replacing the disc with another manifold. By \cite[Theorem D]{ahatblock}, the image of the orbit map $\pi_0(\diff(M))\to\pi_0(\calR^+(M))$ contains an element of infinite order for every manifold $M$ of dimension $d\ge7$ and $d\equiv3(4)$ that has a nontrivial rational Pontryagin class. The mapping torus of one of these diffeomorphisms however is not oriented nullbordant.} However, when working with concordance classes instead, we have the following result pointing towards an affirmative answer of \pref{Question}{q:isotopic}.

\begin{thm}\label{thm:concordance}
Let $M$ be a $d$-dimensional, closed, oriented, totally nonspin manifold with $D\subset M$ an embedded disk and let $g\in\calR^+(M\setminus \inn(D))_{h_\circ}$. Then the metrics $g_j\cup g$ and $g_i\cup g$ are concordant for all $i,j\in\bbN$.
\end{thm}

\begin{proof}
Now let $W_n$, $G_n$ and $g_n$ be as above. We define 
\[W_{M,n}\coloneqq \Bigl(W_n\setminus \bigl(\inn(D_{-}^d)\times[0,1]\bigr)\Bigr)\cup_{S^{d-1}\times[0,1]} \Bigl(\bigl(M\setminus \inn(D)\bigr)\times[0,1]\Bigr)\] 
and let $G_{M,n}\coloneqq G_n\cup g+dt^2$ and $g_{M,n}\coloneqq g_n\cup g$. We will show that $g_{M,n}$ is concordant to $g_M\coloneqq g_{M,0}$. Note that $\beta_n$ is oriented cobordant to a simply connected, oriented, (totally) nonspin manifold $\beta_n^\orient$, for example take $\beta_n^\orient\coloneqq\beta_n\#(\cp{2}\times S^{d-4})$. By \cite[Corollary C]{gl80a}, $\beta_n^\orient$ admits a psc-metric and by \pref{Lemma}{lem:gajer}, there is a psc-metric on $W_{M,n}^\orient\coloneqq M\times[0,1]\#\beta_n^\orient$, which restricts to $g_M$ on both boundary components. Since $W_{M,n}$ is $\so\times B\pi_1(M)$-cobordant to $W_{M,n}^\orient$ relative to the boundary and the stabilised tangential $2$-type of $W_{M,n}$ is given by $\so\times B\pi_1(M)$, \pref{Lemma}{lem:gajer} further implies that $W_{M,n}$ also admits a psc-metric $G$ restricting to $g_M$ on both boundary components. By gluing the metric $G^{\op}$ obtained by flipping $G$ onto $G_{M,n}$, we obtain a psc-metric on $W\coloneqq\overline{W_{M,n}^\op}\cup W_{M,n}$ restricting to $g_M$ and $g_{M,n}$ on the respective boundary components. But $W$ is the double of $W_{M,n}$ and hence $\so\times B\pi_1(M)$-cobordant to $M\times[0,1]$ relative to the boundary and again by \pref{Lemma}{lem:gajer}, there is a concordance between $g_M$ and $g_{M,n}$.
\end{proof}

\begin{rem}
\begin{enumerate}
	\item As mentioned above, isotopic psc-metrics are concordant. The converse to this is a wide open conjecture. An affirmative solution to that conjecture would imply that \pref{Theorem}{thm:concordance} also holds for isotopy classes and hence yields an affirmative answer to \pref{Question}{q:isotopic}.
	\item To the best of the author's knowledge, all known components of $\calR^+(D^d)_{h_\circ}$ contain one of the metrics from \pref{Theorem}{main:hitchin} or \pref{Theorem}{thm:concordance}. One might be tempted to believe that all psc-metrics on $D^d$ restricting to the round metric on the boundary become concordant or even isotopic after gluing them into a totally nonspin manifold. \pref{Theorem}{main:hitchin} and \pref{Theorem}{thm:concordance} could be seen as evidence to support this believe. 
\end{enumerate}
\end{rem}

\begin{changemargin}{-2cm}{-2cm}
\bigskip
\printbibliography
\bigskip
\end{changemargin}

\end{document}